\documentclass{article}

\usepackage{amsmath,url,amssymb,amsthm}

\newcommand{\ZZ}{\mathbf{Z}}
\newcommand{\QQ}{\mathbf{Q}}
\newcommand{\HH}{\mathbf{H}}
\newcommand{\CC}{\mathbf{C}}
\DeclareMathOperator{\SL}{SL}

\newcommand{\Magma}{{\sc Magma}}
\newcommand{\Sage}{{\sc Sage}}

\newtheorem{theorem}{Theorem}
\newtheorem{definition}[theorem]{Definition}
\newtheorem{conjecture}[theorem]{Conjecture}
\newtheorem{remark}[theorem]{Remark}
\newtheorem{question}[theorem]{Question}

\title{On generalized modular forms supported on cuspidal and elliptic points}
\author{L J P Kilford and Wissam Raji}

\begin{document}
\maketitle

\begin{abstract}
In this paper, we extend previous results to prove that generalized modular
forms with rational Fourier expansions whose divisors are supported only at
the cusps and certain other points in the upper half plane are actually
classical modular forms. We discuss possible limitations to this extension and
pose questions about possible zeroes for modular forms of prime level.
\end{abstract}
\section{Introduction}

In a paper about the zeroes of Eisenstein series~\cite{modular-forms-zeroes-quote}, we have the following quote:
\begin{quote}
In the ample theory of classical modular forms, little attention seems to have been paid to their zeros.
\end{quote}
In this note we will attempt to remedy this lack by doing  explicit
computations to find the zeroes of classical modular forms. This is motivated by recent results in the theory of generalized modular forms which use knowledge of the position of the zeroes to prove classification results.

\subsection{Previous work}
The standard application of knowledge of the zeroes of modular forms in the
literature is to compute the dimension of the space of modular forms for the
full modular group~$\SL_2(\ZZ)$, as in Section~3 of Chapter~7
of~\cite{serre-cours}. For spaces of low weight it is possible to locate all
of the zeroes which gives a sharp bound on the dimension, which leads to a
proof of the explicit dimension formula via the valence formula.

A more recent number-theoretic application comes from the $p$-adic world,
where one moves between spaces of overconvergent modular forms of weight~0 and
higher weights by multiplying or dividing by Eisenstein series. In order for
these maps to go from holomorphic forms to holomorphic forms one needs to be
sure of where the zeroes are; see Section~2 of~\cite{buzzard-kilford} for a
recent example of this, and~\cite{kilford-mcmurdy} for a forthcoming paper
where this is important.

There has been some work done on the zeroes of Eisenstein series for genus~0 congruence subgroups; this goes back to~\cite{rankin-swinnerton-dyer} which gives an extremely short and elegant proof is given that the zeroes of~$E_k(z)$ in the standard fundamental domain all lie on the circle $|z|=1$. A more complicated generalization to other level~1 modular forms is given in~\cite{getz}. There are also generalizations to other genus~0 Fuchsian groups in~\cite{hahn}, \cite{miezaki-nozaki-shigezumi}, \cite{shigezumi} and~\cite{garthwaite-et-al}.

However, for the applications we have in mind for this paper, these results are insufficient, for two reasons. Firstly, many of them give asymptotic results, and we will need to know exact locations of zeroes. This is not a critical failure, as there are formulae for zeroes in the literature.

More importantly, we will want to study congruence subgroups of genus at least~1 for reasons that will become clear later. The standard methods in the literature do not seem to generalize away from genus~0, so in this paper we will use explicit arguments and computations to locate the zeroes of modular forms of interest to us.
\subsection{$\eta$-quotients}
There are certain classical modular forms which have explicitly known
zeroes. The best-known examples of this are functions of the Dedekind
$\eta$-product, which we recall has Fourier expansion:
\[
\eta(z)=q^{1/24}\prod_{n=1}^\infty(1-q^n),\text{ where }q:=\exp^{2\pi i z} \text{ and }z \in \HH.
\]
It is well-known that there is a unique normalized cusp form of weight~12 for the full modular group $\SL_2(\ZZ)$, which has Fourier expansion at~$\infty$ given by
\[
\Delta(q)=\eta(q)^{24}=q\prod_{n=1}^\infty (1-q^n)^{24} = \sum_{n=1}^\infty \tau(n) q^n.
\]
The values of the Fourier coefficients~$\tau(n)$ are known as the \emph{Ramanujan $\tau$ function}, and have been extensively studied.

By an analysis of the valence formula in weight~12 (see Equation~\eqref{valence-formula}), we see that~$\Delta$ must have a unique zero at the cusp~$\infty$ and no other zeroes on the upper half plane.

We can also consider $\eta$-quotients of higher level, for congruence subgroups of~$\SL_2(\ZZ)$.  If~$N$ is a positive integer, we can show that an $\eta$-quotient is actually a classical modular function for~$\Gamma_0(N)$ by using Proposition~3.2.1 of~\cite{ligozat}, which says that an $\eta$-quotient~$\prod_{d|N} \eta(q^d)^{r_d}$ is a modular function of weight~0 for~$\Gamma_0(N)$ if
\begin{enumerate}
\item $\sum_{d|N} r_d \cdot d \equiv 0 \mod 24$,
\item $\sum_{d|N} r_d \cdot N/d \equiv 0 \mod 24$,
\item $\sum_{d|N} r_d = 0$ and
\item $\prod_{d|N} (N/d)^{r_d} \in \QQ^2$.
\end{enumerate}
This is known as ``Ligozat's criterion'' in the literature. There is a similar statement for higher weight forms.

A good reference for explicitly computing with $\eta$-products is Chapter~7 of Ken McMurdy's thesis~\cite{mcmurdy-thesis}. We will give an example based on one given in Section~7.4; we can also compute with $\eta$-quotients using {\sc Sage}~\cite{sage}.

We can define an~$\eta$-quotient for the genus~1 congruence subgroup~$\Gamma_0(11)$ by
\[
H_{11}:=\left(\frac{\eta(q)}{\eta(q^{11})}\right)^{12}.
\]
Using Ligozat's criterion we see that~$H_{11}$ is a classical modular function of weight~0 for~$\Gamma_0(11)$.

Now we can see by inspection of the Fourier expansion at~$\infty$ that this function has a pole of order~5 at the cusp~$\infty$, and as~$\Delta$ is non-zero on the upper half plane and~$H_{11}$ is a weight~0 modular function then it must have a zero of order~5 at the cusp~0. This tells us that its divisor is~$5(0)-5(\infty)$.

To deal with more complicated examples, where there are more primes dividing
the level, we need to consider pullbacks of maps between modular curves, which
are covered in detail in~\cite{mcmurdy-thesis}; this provides a good concrete
and algorithmic implementation of Section~2.4 of~\cite{shimura-automorphic}.
\subsection{Other examples}
We recall that an \emph{elliptic point~$P$} for a congruence subgroup~$\Gamma$ is a point of the upper half-plane with a nontrivial stabilizer group under M\"obius transformations. For the full modular group, the non-equivalent elliptic points can be shown to be~$i$ and~$\omega:=\frac{-1+\sqrt{-3}}{2}$.

In some other cases we can explicitly locate the zeroes of modular forms; this often uses the valence formula which we recall is (for~$f$ a nonzero modular form for a subgroup~$G$ of the modular group)
\begin{equation}
\label{valence-formula}
\sum_{\text{cusps of }G}v_c(f) +\frac{1}{2}v_i(f) +\frac{1}{3}v_\omega(f) +\sum_{p \in \Gamma\backslash\HH}v_p(f) = \frac{k \cdot [\SL_2(\ZZ):G]}{12},
\end{equation}
where the second sum is over elements of the fundamental domain of~$G$ which are not equivalent to the elliptic points~$i$ and~$\omega:=(-1+\sqrt{-3})/2$. It is a standard exercise to prove this using contour integrals around a fundamental domain; see for instance Theorem~3 of Chapter~VII.3 of~\cite{serre-cours} for a proof.

For instance, the level~1 Eisenstein series~$E_4$ and~$E_6$ are well-known to have a single zero each within the fundamental domain; from the residue formula, we can see that~$E_4(\omega)=E_6(i)=0$.

Similar calculations allow us to pin down the zeroes for a handful of cases, especially when we are dealing with modular forms for genus~0 subgroups. This can then be used for higher genus subgroups, because if~$f$ is a modular form for a genus~0 subgroup~$\Gamma$ then it is also a modular form for finite index subgroups of~$\Gamma$, which can have higher genus.

We care about the genus of the subgroup because in Knopp and Mason~\cite{knopp-mason} they prove a result (Theorem 2) that says that a generalized modular form~$F$ for a finite index subgroup of weight 0 corresponds to a classical modular form~$L$ of weight 2; further, if~$L$ is a cusp form, then~$F$ is an entire form. This is in stark contrast to the classical situation, where an entire classical modular form of weight~0 must be constant (Choose a point~$z_0$ in the upper half-plane. If~$f(z_0)=a$ then~$f(z)-a$ is a modular form with a zero at~$z_0$ and by the valence formula it must be identically zero, so~$f(z)=a$ as required).

\section{Extending the classification theorem}
We will need to consider the concept of a \emph{generalized} modular form.
\begin{definition}
Let~$\Gamma$ be a finite-index subgroup of the modular group. A generalized modular form of weight~0 for~$\Gamma$ is a holomorphic function~$f:\HH\rightarrow\CC$ which satisfies
\[
f(\gamma\tau)=\chi(\gamma)f(\tau)\text{ for }\gamma \in \Gamma,
\]
for some not necessarily unitary character~$\chi:\Gamma\rightarrow\CC^\star$, and which is also meromorphic at the cusps of~$\Gamma$.
\end{definition}
We can extend this definition to integral and half-integral weight generalized modular forms in the obvious way, by multiplication by classical forms of known weight.

As mentioned above, classical cusp forms of weight~2 for~$\Gamma$ correspond
to generalized modular forms of weight~0 with cuspidal divisor, so there are
many nontrivial examples of generalized modular forms which are \emph{not}
classical modular forms.

Here is an explicit example. Let~$f$ be an integral of weight~0, which
satisfies the equation
\begin{equation}
\label{gmf-period}
f(\gamma z) = f(z) + c,
\end{equation} 
for~$\gamma \in \Gamma$ and for some period~$c$. If we
exponentiate this, then we obtain a nontrivial entire generalized modular form of
weight~0 (in other words, it is not constant, so it is not a classical modular
form).

In~\cite{kohnen-mason-applications}, the following conjecture is made about generalized modular forms satisfying certain conditions.
\begin{conjecture}[\cite{kohnen-mason-applications}, Conjecture~(2)]
\label{km-conjecture}
Let~$N$ be a positive integer. If~$f$ is a generalized modular form for~$\Gamma_0(N)$ with Fourier expansion at~$\infty$ defined over~$\ZZ$ and the smallest nonzero Fourier coefficient normalized to be~1, then~$f$ is a classical modular form.
\end{conjecture}

In this section we will consider results which prove Conjecture~\ref{km-conjecture} in special cases, when we restrict the divisor of the modular forms we consider.

We recall the definition of the second Bernoulli function:
\[
P_2(t)=\{t\}^2-\{t\}+1/6,
\]
where~$\{t\}:=t-\lfloor t \rfloor$ is the fractional part of the rational number~$t$.

Finding the zeroes of a modular form is of interest to us because
in~\cite{raji-gmf} there is the following result, which we will call the
\emph{classification theorem}:
\begin{theorem}[\cite{kohnen-mason-new-paper}, Corollary~2]
\label{wissam-theorem}
Let~$N$ be a positive integer and let~$f$ be a generalized modular form for~$\Gamma_0(N)$. 
If the Fourier coefficients of~$f$ at the cusp~$\infty$ are rational and $p$-integral for all but a finite number of primes~$p$ and that the poles and zeroes of~$f$ are supported at the cusps, then~$f$ is a \emph{classical} modular form.
\end{theorem}

We note that in the literature that a classical modular function which
satisfies this condition is said to have a \emph{Heegner divisor}.

We also note that there are classification theorems for other standard
congruence subgroups such as~$\Gamma_1(N)$ and~$\Gamma(N)$. In this paper we
are mostly concerned with $\Gamma_0(N)$ as the other groups do not
have elliptic points for $N > 2$.

\begin{remark} 
This important theorem generalizes other work by several authors.
If the congruence subgroup $\Gamma_0(N)$ satisfies certain conditions, such as $N$ being
squarefree, then there are proofs of the theorem which explicitly create a classical
modular form with rational Fourier expansion and with the correct poles and
zeroes supported at the cusps. A good example of this is Theorem~3 of~\cite{raji-gmf}.
\end{remark}
In particular, this result is non-empty for~$N=11$, because by the result of Knopp and Mason quoted above there are nontrivial entire generalized modular forms for~$\Gamma_0(11)$ (which has genus~1); we can use it to show that the form~$H_{11}$ we defined earlier is modular based only on the knowledge of its zeroes and the fact that its Fourier expansion at~$\infty$ is rational.

It is important to note that the conditions in the theorem on the rationality
of the Fourier expansion of $f$ are necessary; if they are weakened then there
are examples of generalized modular forms of weight~0 with no zeroes which are not classical modular forms.

We can also prove the following result:

\begin{theorem}
Let~$F(z)$ be an Eichler integral of weight 0 on a subgroup of genus zero of
the full group.  If the period~$c$ in~\eqref{gmf-period} is pure imaginary or
zero at the parabolic elements of~$\Gamma$, then it is either pure imaginary or~0 for all $\gamma \in \Gamma$.
\end{theorem}
\begin{proof}
Let $G(z)=e^{F(z)}$.  Then $G(z)$ is an entire parabolic generalized modular
form (PGMF) of weight 0 with multiplier system $e^{c}$.  Since every entire PGMF on a
subgroup of genus zero is classical (by Theorem~2 of~$\cite{knopp-mason}$)
then $\mid e^{c}\mid=1$, as required.
\end{proof}

We can extend this result to characterize modular forms which are of the form $E_4^a \cdot E_6^b \cdot F(\eta)$ where~$F(\eta)$ is an $\eta$-quotient. $E_4$ as a modular form for $\SL_2(\ZZ)$ has a unique zero at $\omega$ and $E_4$ as a modular form for $\SL_2(\ZZ)$ has a unique zero at $i$; this means that if we look at them as forms for $\Gamma_0(N)$, say, then they have zeroes at all of the points which are equivalent to $\omega$ or $i$ in the fundamental domain for $\Gamma_0(N)$. 

We can make this more precise using arguments based on Section~2.4 of~\cite{shimura-automorphic}.  For instance, let us take~$N=31$ and consider the forgetful map~$X_0(31)\rightarrow X_0(1)$. We recall that we can view every non-cuspidal point of~$X_0(M)$ as a pair~$(E,C)$ where~$E$ is an elliptic curve and~$C$ is a cyclic subgroup of~$E$ with order~$M$. The forgetful map sends~$(E,C)$ to~$(E,\mathbf{1})$, where~$\mathbf{1}$ is the trivial subgroup. If~$N$ is prime then this map is unramified at the cusp~$\infty$ and ramified~$N$ times at the cusp~$0$.

The elliptic points of~$X_0(N)$ lie over the elliptic points for~$\SL_2(\ZZ)$;
they are the points with ramification index~1. There are also points of~$\HH$ with ramification index~2 (over~$i$) and~3 (over~$\omega$); these are not elliptic points for~$\Gamma_0(N)$.

If~$N=31$, then the congruence subgroup~$\Gamma_0(31)$ has index~32
in~$\SL_2(\ZZ)$, so the point~$\omega$ splits into two elliptic points with
ramification index~1 each, and~10 non-elliptic points each with ramification
index~3. This means in particular that~$E_4$, viewed as a modular oldform
for~$G:=\Gamma_0(31)$, has simple zeroes at both of the elliptic points.

This will be a nontrivial example as $G$ has genus~2, so there are known to be nontrivial entire generalized modular forms. $G$ has index 32 in $\SL_2(\ZZ)$
and two inequivalent elliptic points ($P$ and $Q$), so the valence formula looks like
\[
v_0(f) + v_\infty(f) + \frac{1}{3}\cdot [v_P(f) + v_Q(f)] + \cdots = \frac{32k}{12}= \frac{8k}{3},
\]
where the dots are a sum over the other inequivalent points of the fundamental
domain.

It is well-known that modular forms of weight at least 2 can be computed efficiently using algorithms based upon the arithmetic of modular symbols. This gives us an algorithm to compute weight~1 modular forms, which we will sketch here (this technique goes back to the work of Buhler in~\cite{buhler-icosahedral}, where he computes the icosahedral cusp form of weight~1 and level~800).

Let~$f$ be a modular form of weight~1 that we wish to compute and let~$\{g_i\}$ be a set of known modular forms, such as Eisenstein series. We find the images~$f \cdot g_i$, which lie in spaces with higher weight, and because the image spaces have weight at least~2, we can compute them efficiently. This has been implemented in the computer algebra system {\sc Magma}~\cite{magma}, and could also be done by hand in {\sc Sage}.

So for $k=1$ we find using {\sc Magma} that there is a unique cusp form $f$ of weight 1 with Fourier expansion at infinity given by
\[
q - q^2 - q^5 - q^7 + q^8 + q^9 + q^{10} + O(q^{12})
\]
and because we know it has zeroes at~0 and~$\infty$ and the weighted sum of the zeroes is~$8/3$ we see that it must have unique zeroes at 0 and infinity and two zeroes at the elliptic points. 

We can prove that~$f$ has a simple zero~$P$ and a simple zero at~$Q$ by computation. Above we showed that~$E_4$ has simple zeroes at~$P$ and~$Q$. We now use {\sc Magma} to check that~$\Delta \cdot E_4 / f$ is a modular form in~$M_{15}(\Gamma_0(31),\left(\frac{\cdot}{31}\right))$; this space of modular forms has dimension~39 so the computation takes a few seconds on a modern computer. The fact that this quotient is holomorphic means that~$f$'s zeroes at the two elliptic points have order at most~1, so they must have order exactly~1.

So if $g$ is a generalized modular form for $\Gamma_0(31)$ with zeroes and poles supported at the cusps and the elliptic points, and the zeroes or poles at the elliptic points have the same order, then $g$ is a classical modular form, because we can multiply it by a power of $f$ to kill off the zeroes/poles at $P$ and $Q$, and then we use Theorem~\ref{wissam-theorem} to show that it's classical. 

We could also use the fact that the Atkin-Lehner operator~$w_{31}$ exchanges
the two elliptic points of~$G$ and acts by multiplication by~$-1$ on~$f^2$,
so~$f^2$ has double zeroes at both infinity and~0 and zeroes of the same order at each of the elliptic points. We have now located all of the zeroes, and as above we can extend Theorem~\ref{wissam-theorem} to include the case where the divisor is supported at the elliptic points with equal vanishing at both of them.

In a similar way, we can find the zeroes of forms~$f$ for the
groups~$\Gamma_0(N)$ given in Figure~\ref{gamma-0-n-table} by either using the
Atkin-Lehner operator to find the zeroes at the cusps or by showing that
quotients of the form~$E_4 \cdot E_6 \cdot \Delta(q)\Delta(q^N)/f$ are actually
modular forms. If the latter holds, then we know that~$f$ can only vanish at
the cusps and points above the elliptic points for~$\SL_2(\ZZ)$, and by
further calculations of this type we can show that~$f$ actually vanishes at
the elliptic points of~$\Gamma_0(N)$. We summarize our results in Figure~\ref{gamma-0-n-table}.

\begin{figure}[ht!]
\begin{center}
\begin{tabular}{|c|c|c|c|c|c|c|c|}
\hline
$N$ & 0 & $\infty$ & $C_1$ & $C_2$ & $i$ & $\omega$ \\
\hline
1  & -- & 1 & -- & -- & 1 & 0 \\
1  & -- & 1 & -- & -- & 0 & 1 \\
2  & 1 & 1 & -- & -- & 1 & 0 \\
3  & 1 & 1 & -- & -- & 2 & 0 \\
5  & 1 & 1 & -- & -- & 2 & 0 \\
7  & 1 & 1 & -- & -- & 0 & 2 \\
13 & 3 & 1 & -- & -- & 0 & 2 \\
17 & 1 & 1 & -- & -- & 2 & 0 \\
19 & 1 & 1 & -- & -- & 0 & 4 \\
21 & 1 & 1 & 1  & 1  & 0 & 4 \\
26 & 2 & 2 & 1  & 1  & 2 & 0 \\
29 & 2 & 2 & -- & -- & 2 & 0 \\
31 & 2 & 2 & -- & -- & 0 & 4 \\
34 & 3 & 3 & 1  & 1  & 2 & 0 \\
39 & 3 & 3 & 1  & 1  & 0 & 4 \\
41 & 3 & 3 & -- & -- & 2 & 0 \\
49 & 1 & 1 & \dag & \dag & 0 & 4\\
50 & 2 & 2 & 1\dag & 1\dag & 2 & 0\\
\hline
\end{tabular}
\caption{Cuspforms defined over~$\QQ$ with zeroes only at the cusps and the elliptic points}\label{gamma-0-n-table}
\end{center}
\end{figure}

In this table the numbers given are the orders of vanishing at the given
places, where~$C_1$ and~$C_2$ are cusps which are not equivalent to~0
or~$\infty$ and~$i$ and~$\omega$ are elliptic points above~$i$ or~$\omega$
respectively. The \dag entries for~$N=49$ and~$N=50$ indicate that the forms
there have simple zeroes for all of the cusps which are not~0 or~$\infty$.

For completeness, we include modular forms for congruence subgroups of
genus~0 and prime level, although these are not of direct interest in our
current work. The entries for~$N=1$ and~$N=2$ are unusual because there is
only one elliptic point for these levels, and for level~1 there is only one cusp.

We note here that the modular form~$f$ we considered above has no ``free
zeroes''. This term was used in~\cite{petersson-nullstellen} to mean zeroes
that are not forced; for instance, a cusp form must have at least one zero at
each of the cusps, and modular forms for certain congruence subgroups which
have certain weights must have zeroes at the elliptic points of those
subgroups; these are all examples of non-free zeroes. Examples of unforced
zeroes in this sense are those for~$N=34$; the valence formula
does not force there to be zeroes at elliptic points, and the Atkin-Lehner operator
does not act as a scalar so the extra zeroes at the cusp~$0$ are unforced.

There are also cases where we can explicitly identify all of the zeroes except one. If we have $N=28$ then we find that~$\Gamma_0(28)$ has index~8, the space of weight~2 cusp forms has dimension~2 and there are 6 cusps, so there is a cusp form with rational Fourier expansion beginning~$q^2 - q^4 - 2q^6 + q^8 + O(q^{12})$ which has two zeroes at~$\infty$, one zero at each of the other cusps, and one other zero which we will call~$Z_0$. 

Similarly, if~$N=47$ then there is a modular form of weight~1 with
character the Legendre symbol modulo~47, which has two zeroes
at infinity, one zero at the cusp~0, and one other zero ($Z_1$), and has
rational Fourier coefficients.  So any generalized modular form for
$\Gamma_0(28)$ or $\Gamma_0(47)$ which vanishes only at infinity, 0 or~$Z_i$
(for the correct~$i$) and has rational Fourier coefficients is a classical modular form.

These particular lines of attack will not work on many more general cases because we are relying on a numerical coincidence; the number of cusps of~$\Gamma_0(N)$ goes up as a function of the number of divisors of~$N$, whereas the number of zeroes of a form~$f$ for~$\Gamma_0(N)$ goes up as a function of~$N$; as~$N$ increases, there will be too many zeroes to be accounted for by the cusps and elliptic points.

\section{Limitations of the classification theorem}

In this section we will consider some conjectural limitations of the
extension of the classification theorem that we proposed in the last
section. We will take~$N=p \ge 17$ to be prime here, and~$f$ to be
a modular form for~$\Gamma_0(p)$ of weight~2.

In this situation, we have two cusps, traditionally represented by~$\infty$
and~$0$. We also have either~2 or~4 elliptic points (we will ignore the case
where we have no elliptic points, which is when~$p\equiv 11\mod12$).

In~\cite{ahlgren-masri-rouse}, the question of how large the maximal order of
vanishing of a modular form at~$\infty$ can be is raised; for weight~2 forms
of prime level~$\Gamma_0(N)$, this is known to be the genus of~$X_0(N)$, and
this is conjectured to be true for squarefree~$N$ (this is \emph{not} true for non-squarefree levels, the smallest example
being~$N=54$). Because the Atkin-Lehner operator exchanges the cusps in prime
levels, this means that the maximum order of vanishing at the cusp~$0$ must
also be the genus of~$X_0(N)$. Finally, we can see that there is a unique form with rational coefficients with maximal
order of vanishing at a given cusp (if there were two such, their difference
vanishes to a higher order at that cusp, which is a contradiction).

A sensible type of weight~2 form~$f$ to look for is one which has a zero of maximum order
at both~$\infty$ and~$0$; from the previous paragraph we see that this must be
an eigenform for the Atkin-Lehner operator. We must now consider the different
congruence classes that~$p$ can lie in separately.

In the case that~$p \equiv 7\mod12$ the elliptic points lie above~$\omega$,
and the valence
formula tells us that~$f$ must have~$4/3$ zeroes away from the cusps, so it
must have a zero at an elliptic point, and the action of Atkin-Lehner can be
shown by computation to switch the elliptic points, so it must have a zero at the other point, so it must
have four zeroes at the elliptic points, and as it is an eigenform for
Atkin-Lehner it must have two zeroes at each point.

In the case~$p \equiv 5\mod12$, the elliptic points lie above~$i$, and the
valence formula tells us that there will be a a zero or zeroes of total
weight~1. As~$f$ is an eigenform for the Atkin-Lehner operator, there are two
possibilities; either $f$ has a simple zero at both of the elliptic points
over~$i$ (which each have weight~$1/2$), or it has a simple zero at one of the fixed points of the
Atkin-Lehner transformation ($z^2=1/p$).

In the final case, where~$p \equiv 1\mod12$, there are four elliptic points, two above
both~$i$ and two above~$\omega$, and so the analysis is simply a combination of the two
paragraphs above. The extra zeroes not accounted for by the cusps have
weight~$7/3$, so there are clearly elliptic points over~$\omega$ and there may
be elliptic points over~$i$.

Finding such a form can be done experimentally by computing a rational basis for the
space $S_2(\Gamma_0(p))$ and then checking to see if the unique form with
largest vanishing at~$\infty$ is an eigenform for~$w_p$. This can be done
efficiently with \Magma{} or \Sage{}.

However, forms with maximum vanishing at the cusp~$\infty$ are very rarely eigenforms for~$w_p$; if the Atkin-Lehner
operator acts on the space~$S_2(\Gamma_0(p))$ as multiplication by a scalar, then they are
forced to be, but otherwise this doesn't often happen. Table~5 of
Antwerp~IV~\cite{antwerp-IV} (this has been extended by Kohel~\cite{antwerp-Iv-kohel}) gives the dimension of eigenspaces for
the Atkin-Lehner involution; we see that in this table there are only a handful of these
where~$w_p$ acts as multiplication, all of which are discussed
in the section above. It is plausible that our table actually gives a complete list.

We also notice that the list given in Figure~\ref{gamma-0-n-table} is, with
one interesting exception, contained within the list given in Theorem~2
of~\cite{ogg-hyperelliptic-modular-curves}; the table in Ogg gives the levels~$N$ for which
the modular curve~$X_0(N)$ is hyperelliptic and the Atkin-Lehner involution is
also the hyperelliptic involution~$\backsim$. These are also the integers~$N$ for
which~$X_0(N)/\!\!\backsim$ has genus~0, which play an interesting and surprising
r\^ole in the subject of ``Monstrous
moonshine''~\cite{borcherds-moonshine}. This also lends weight to our
suggestion that we have a complete list of such forms.

The exception is~$N=34$, which is not
hyperelliptic, and where the Atkin-Lehner involution does \emph{not} act as
multiplication by a scalar, but there is a modular form with maximum vanishing at the
cusp~$\infty$ which is an eigenform for the Atkin-Lehner operator.

We therefore ask the following question:
\begin{question}
Let~$p$ be a prime number which is not congruent to~11 modulo~12. If~$f$ is a
modular form for~$\Gamma_0(p)$ which vanishes at the cusps and at the elliptic
points and nowhere else, is its level one of those given in Table~\ref{gamma-0-n-table}?
\end{question}

Clearly we can ask similar questions for non-prime levels, and we expect that
the answer will be similar in both cases. 

We can also ask a weaker question:
\begin{question}
Let~$P$ be the set of prime numbers which are not congruent to~11 modulo~12,
and let~$Q$ be the set of primes~$p$ in~$P$  for which there exists a modular
form~$f$ for~$\Gamma_0(p)$ which vanishes at the cusps and at the elliptic
points and nowhere else. Does~$Q$ have density~0 in~$P$?
\end{question}

As before, we can reasonably generalize this to arbitrary squarefree natural numbers~$N$;
for non-squarefree numbers the situation is less clear, as we noted above that
the maximum order of vanishing can be larger than the genus of~$X_0(N)$.

\section{Acknowledgements}

The first-named author would like to thank the American University of Beirut
for its hospitality during the Conference on Modular Forms and Related Topics in June 2009 and the Heilbronn
Institute and the University of Bristol for supporting his research. He would
also like to thank Ken McMurdy for many helpful conversations, Kamal
Khuri-Makdisi for his help during a research visit to the University of
Bristol, and William Stein for helpful conversations. Thanks are also due to Robin Chapman for a helpful suggestion given via \url{MathOverflow.net}.

The second-named author would like to thank the American University of Beirut for its support for his research.


\end{document}